\documentclass{amsart}
\usepackage{amscd,amssymb}
\usepackage[all]{xy}
\begin{document}

\makeatletter
\def\revddots{\mathinner{\mkern1mu\raise\p@
  \vbox{\kern7\p@\hbox{.}}\mkern2mu
  \raise4\p@\hbox{.}\mkern2mu\raise7\p@\hbox{.}\mkern1mu}}
\makeatother

\newcommand{\Inn}{\operatorname{Inn}\nolimits}
\newcommand{\Der}{\operatorname{Der}\nolimits}
\newcommand{\Hom}{\operatorname{Hom}\nolimits}
\renewcommand{\Im}{\operatorname{Im}\nolimits}
\newcommand{\Ker}{\operatorname{Ker}\nolimits}
\newcommand{\Coker}{\operatorname{Coker}\nolimits}
\newcommand{\rad}{\operatorname{rad}\nolimits}
\newcommand{\rrad}{\mathfrak{r}}
\newcommand{\Ext}{\operatorname{Ext}\nolimits}
\newcommand{\id}{{\operatorname{id}\nolimits}}
\newcommand{\pd}{{\operatorname{pd}\nolimits}}
\newcommand{\End}{\operatorname{End}\nolimits}
\renewcommand{\L}{\Lambda}
\newcommand{\Y}{{\mathcal Y}}
\renewcommand{\P}{{\mathcal P}}
\newcommand{\HH}{\operatorname{HH}\nolimits}
\newcommand{\mo}{\mathfrak{o}}
\renewcommand{\mod}{\operatorname{mod}\nolimits}
\newcommand{\mt}{\mathfrak{t}}
\newcommand{\NonTip}{\operatorname{NonTip}\nolimits}
\newcommand{\Tip}{\operatorname{Tip}\nolimits}
\newcommand{\tip}{\operatorname{tip}\nolimits}
\newcommand{\Bdy}{\operatorname{Bdy}\nolimits}
\newcommand{\typ}{\operatorname{typ}\nolimits}
\newcommand{\kar}{\operatorname{char}\nolimits}

\newtheorem{lem}{Lemma}[section]
\newtheorem{prop}[lem]{Proposition}
\newtheorem{cor}[lem]{Corollary}
\newtheorem{thm}[lem]{Theorem}
\newtheorem{bit}[lem]{}
\theoremstyle{definition}
\newtheorem{defin}[lem]{Definition}
\newtheorem*{remark}{Remark}
\newtheorem{example}[lem]{Example}

%\baselineskip=18pt

%\begin{titlepage}\phantom{|}\vspace{-0.75in}

%\vspace*{1.5in}

\title[Distinguishing derived equivalence classes \dots]
{Distinguishing derived equivalence classes using the second Hochschild cohomology group}
\author[Al-Kadi]{Deena Al-Kadi}

\address{Deena Al-Kadi\\
Department of Mathematics\\
Taif University\\
Taif\\
Saudia Arabia}

\begin{abstract}
\sloppy In this paper we study the second Hochschild cohomology group
of the preprojective algebra of type $D_4$ over an algebraically closed
field $K$ of characteristic 2. We also calculate the second Hochschild
cohomology group of a non-standard algebra which arises as a socle
deformation of this preprojective algebra and so show that the two algebras
are not derived equivalent. This answers a question raised by Holm and
Skowro\'nski.
\end{abstract}

\date{\today}

\maketitle
\section*{Acknowledgements}
I thank Nicole Snashall for her helpful comments.

\section*{Introduction}

The main work of this paper is in determining the second Hochschild
cohomology group ${\HH}^2(\L)$ for two finite dimensional algebras $\L$ over
a field of characteristic 2 in order to show that they are not derived
equivalent. We let ${\mathcal A}_1$ denote the preprojective algebra of type
$D_4$; this is a standard algebra. We introduce, in Section \ref{ch1}, the
algebra ${\mathcal A}_2$ by quiver and relations; this is a non-standard algebra which is socle
equivalent to ${\mathcal A}_1$, in the case where the underlying field has
characteristic 2. This work is motivated by the question asked by Holm and
Skowro\'nski as to whether or not these two algebras are derived equivalent.

The algebras ${\mathcal A}_1$ and ${\mathcal A}_2$ are selfinjective
algebras of polynomial growth. The main result of this paper (Corollary
\ref{cor}) shows that they are not derived equivalent. This answer to the
question of Holm and Skowro\'nski enabled them to complete their
derived equivalence classification of all symmetric algebras of polynomial growth in \cite{HS}.
We note that \cite{D} showed that the second Hochschild cohomology group
could also be used to distinguish between derived equivalence classes of
standard and non-standard algebras of finite representation type.

Throughout this paper, we let $\L$ denote a finite dimensional algebra over
an algebraically closed field $K$. We start, in Section \ref{ch1}, by giving
the quiver and relations for the two algebras ${\mathcal A}_1$ and ${\mathcal
A}_2$, and recall that we are interested only in the case when $\kar K = 2$.
(We write our paths in a quiver from left to right.)
In Section \ref{ch2}, we give a short description of the projective
resolution of \cite{GS} which we use to find ${\HH}^2(\L)$. The remaining
two sections determine
${\HH}^2(\L)$ for $\L = {\mathcal A}_1, {\mathcal A}_2$. As a consequence,
we show in Corollary \ref{cor} that $\dim\HH^2({\mathcal A}_1) \neq
\dim\HH^2({\mathcal A}_2)$ and hence these two algebras are not derived
equivalent.

\section{The algebras ${\mathcal A}_1$ and ${\mathcal A}_2$ }\label{ch1}

In this section we describe the algebras ${\mathcal A}_1$ and ${\mathcal
A}_2$ by quiver and
relations. We assume that $K$ is an algebraically closed field
and $\kar K$ = 2. The standard algebra ${\mathcal A}_1$ is the
preprojective algebra of type $D_4$, and we note that it was shown in
\cite{ES2} that, in the case when $\kar K \neq 2$, we have
${\HH}^2({\mathcal A}_1) = 0$. We will see that this is
in contrast to the $\kar K = 2$ case.

The algebra ${\mathcal A}_1$ is the given by the quiver ${\mathcal
Q}$:
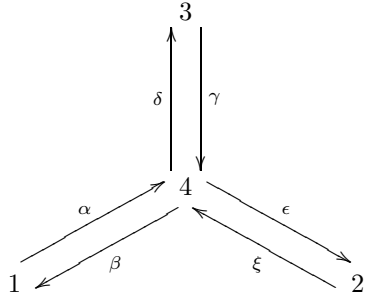
$$\xymatrix{
& &3  \ar@<1.3 ex>[dd]^{\gamma}&\\
&&&\\
& &4 \ar@<1.3 ex>[dll]^{\beta} \ar@<1.3 ex>[drr]^{\epsilon} \ar@<1.3 ex>[uu]^{\delta}&\\
1 \ar@<1.3 ex>[urr]^{\alpha} &&&& 2 \ar@<1.3 ex>[ull]^{\xi}
}$$
with relations
$$\beta \alpha + \delta \gamma + \epsilon \xi = 0, \gamma \delta = 0, \xi \epsilon = 0
\mbox{ and }
\alpha \beta = 0.$$

The algebra ${\mathcal A}_2$ is the non-standard algebra given by the same
quiver ${\mathcal Q}$ with relations
$$\beta \alpha + \delta \gamma + \epsilon \xi = 0 , \gamma \delta = 0 , \xi \epsilon = 0,
 \alpha \beta \alpha = 0, \beta \alpha \beta = 0 \mbox{ and } \alpha \beta =
\alpha \delta \gamma \beta.$$

\bigskip

We need to find
a minimal set of generators $f^2$ for each algebra. We start with the algebra ${\mathcal A}_2$. The set $\{\alpha \beta - \alpha \delta \gamma \beta, \xi
\epsilon, \gamma \delta, \beta \alpha + \delta \gamma + \epsilon \xi, \alpha \beta \alpha,
\beta \alpha \beta\}$ is not a minimal set of generators for $I$ where ${\mathcal A}_2 = K{\mathcal Q}/
I$. Let $x  = \beta \alpha + \delta 
\gamma + \epsilon \xi$ and let $y = \alpha \beta - \alpha \delta \gamma \beta$. 
We will show that $\alpha\beta\alpha$ is in the ideal generated by $x, y, 
\gamma\delta, \xi\epsilon$. Using that $\kar K = 2$, we have 
$\alpha \beta \alpha = y\alpha + \alpha \delta \gamma \beta \alpha
= y\alpha + \alpha x\beta\alpha + \alpha(\beta \alpha+ \epsilon\xi)\beta\alpha
= y\alpha + \alpha x\beta\alpha + \alpha \beta \alpha \beta \alpha +
\alpha\epsilon\xi x + \alpha \epsilon \xi (\delta \gamma + \epsilon \xi) 
= y\alpha + \alpha x\beta\alpha + \alpha\epsilon\xi x + 
\alpha \beta \alpha \beta \alpha + \alpha x\delta\gamma +
\alpha (\beta \alpha + \delta \gamma) \delta \gamma + \alpha \epsilon \xi \epsilon \xi
= y\alpha + \alpha x\beta\alpha + \alpha\epsilon\xi x + \alpha x\delta\gamma +
\alpha \epsilon \xi \epsilon \xi + \alpha \beta \alpha \beta \alpha +
\alpha\beta\alpha x + \alpha \beta \alpha (\beta \alpha + \epsilon \xi) + 
\alpha \delta \gamma \delta \gamma
= y\alpha + \alpha x\beta\alpha + \alpha\epsilon\xi x + \alpha x\delta\gamma +
\alpha \epsilon \xi \epsilon \xi + \alpha\beta\alpha x + 
\alpha \beta \alpha \epsilon \xi + \alpha \delta \gamma \delta \gamma$.
However, $\alpha \beta \alpha \epsilon \xi = 
y\alpha\epsilon\xi + \alpha \delta \gamma \beta \alpha \epsilon \xi 
= y\alpha\epsilon\xi + \alpha \delta \gamma x \epsilon \xi 
+ \alpha \delta \gamma (\delta \gamma + \epsilon \xi) \epsilon \xi$.
Thus $\alpha \beta \alpha$ is in the ideal generated by
$x, y, \gamma\delta, \xi\epsilon$. 
Using a similar argument for $\beta\alpha\beta$, we have that 
$I$ is generated by the set
$\{\alpha \beta - \alpha \delta \gamma \beta, \xi \epsilon, \gamma \delta, \beta \alpha + \delta \gamma + \epsilon \xi\}$.
This gives the following result.

\begin{prop}\label{prop1}
For ${\mathcal A}_2$ let
$$f^2_1 = \alpha \beta - \alpha \delta \gamma \beta,$$
$$f^2_2 = \xi \epsilon, \hspace{1.5cm}$$
$$f^2_3 = \gamma \delta, \hspace{1.5cm}$$
$$f^3_4 = \beta \alpha + \delta \gamma + \epsilon \xi.$$
Then $f^2 = \{f^2_1, f^2_2, f^2_3, f^2_4\}$ is a minimal set of generators of $I$ where
${\mathcal A}_2 = K{\mathcal Q} /I$.
\end{prop}

We now consider the algebra ${\mathcal A}_1$.

\begin{prop}\label{prop2}
For ${\mathcal A}_1$ let
$$f^2_1 = \alpha \beta,\hspace{1.5cm}$$
$$f^2_2 = \xi \epsilon, \hspace{1.5cm}$$
$$f^2_3 = \gamma \delta, \hspace{1.5cm}$$
$$f^3_4 = \beta \alpha + \delta \gamma + \epsilon \xi.$$
Then $f^2 = \{f^2_1, f^2_2, f^2_3, f^2_4\}$ is a minimal set of generators for $I'$ where
${\mathcal A}_1 = K{\mathcal Q} /I'$.
\end{prop}

\section{The Projective resolution}\label{ch2}

To find the Hochschild cohomology groups for any finite dimensional algebra
$\L$, a projective resolution of $\L$ as a $\L, \L$-bimodule is needed. In
this section we look at the projective resolutions of \cite{GS} and \cite
{GSZ} in order to describe the second Hochschild cohomology group. Let $K$
be a field and let $\L =K {\mathcal Q}$/$I$ be a finite dimensional algebra
where ${\mathcal Q}$ is a quiver,
and $I$ is an admissible ideal of $K {\mathcal Q}$.
Fix a minimal set $f^2$ of generators for the ideal $I$. 
For any $x \in f^2$, we may write $x =
\sum_{j=1}^{r} c_ja_{1j} \cdots a_{kj} \cdots a_{s_jj}$, where the $a_{ij}$ are
arrows in ${\mathcal Q}$ and $c_j \in K$, that is, $x$ is a
linear combination of paths $a_{1j} \cdots a_{kj} \cdots a_{s_jj}$ for $j =
1, \ldots, r$. We may assume that there are (unique) vertices $v$ and $w$ such
that each path $a_{1j} \cdots a_{kj} \cdots a_{s_j j}$ starts at $v$ and
ends at $w$ for all $j$, so that $x = vxw$. We write $\mo(x) = v$ and $\mt(x) = w.$ Similarly
$\mo(a)$ is the origin of the arrow $a$ and $\mt(a)$ is the terminus of $a$.

In \cite[Theorem 2.9]{GS}, the first 4 terms of a minimal projective
resolution of $\L$ as a $\L,
\L$-bimodule are described:
$$\cdots \rightarrow Q^3 \stackrel{A_3}{\rightarrow} Q^2
\stackrel{A_2}{\rightarrow} Q^1 \stackrel{A_1}{\rightarrow} Q^0
\stackrel{g}{\rightarrow} \L \rightarrow 0.$$
The projective $\L, \L$-bimodules $Q^0, Q^1, Q^2$ are given by
$$Q^0 = \bigoplus_{v, vertex} \L v \otimes v\L,$$
$$Q^1 = \bigoplus_{a, arrow} \L \mo(a) \otimes \mt(a)\L, \mbox { and }$$
$$Q^2 = \bigoplus_{x \in f^2} \L \mo(x) \otimes \mt(x) \L.$$
Throughout, all tensor products are over $K$, and we write
$\otimes$ for $\otimes_K$. The maps $g, A_1$, $A_2$ and $A_3$ are
all $\L, \L$-bimodule homomorphisms.
The map $g: Q^0 \rightarrow \L$ is the multiplication map so is given by  $v
\otimes v \mapsto v$. The map $A_1: Q^1\rightarrow Q^0$ is given by $\mo(a)
\otimes \mt(a) \mapsto \mo(a) \otimes \mo(a) a - a \mt(a) \otimes \mt(a)$
for each arrow $a$.
With the notation for $x \in f^2$ given above, the map $A_2: Q^2 \rightarrow
Q^1$ is given by $\mo(x) \otimes \mt(x) \mapsto
\sum_{j=1}^{r}c_j(\sum_{k=1}^{s_j} a_{1j} \cdots a_{(k-1)j} \otimes
a_{(k+1)j} \cdots a_{s_j j})$, where $a_{1j} \cdots a_{(k-1)j} \otimes
a_{(k+1)j} \cdots a_{s_j j}$ $\in \L \mo(a_{kj}) \otimes \mt(a_{kj})\L$.

In order to describe the projective $Q^3$ and the map $A_3$
in the $\L, \L$-bimodule resolution of $\L$ in \cite{GS}, we need to
introduce some notation from \cite{GSZ}. Recall that an
element $y \in K{\mathcal Q}$ is uniform if there are vertices $v, w$ such
that $y = v y = y w.$ We write $\mo(y) = v$ and $\mt(y) = w$. In \cite{GSZ}, Green,
Solberg and Zacharia show that there are sets $f^n$ in $K{\mathcal Q}$, 
for $n \geq 3$, consisting of uniform
elements $y \in f^n$ such that $y = \sum_{x \in f^{n-1}} x r_x = \sum_{z \in
f^{n-2}} z s_z$ for unique elements $r_x, s_z \in K{\mathcal Q}$ such that
$s_z \in I$. These sets have
special properties related to a minimal projective $\L$-resolution of
$\L/\rrad$, where $\rrad$ is the Jacobson radical of $\L$. Specifically
the $n$-th projective in the minimal projective $\L$-resolution of $\L/\rrad$
is $\bigoplus_{y \in f^n} \mt(y) \L.$

In particular, for $y \in f^3$
we may write $y = \sum f^2_i p_i = \sum q_i f^2_i r_i$ with $p_i, q_i, r_i \in
K{\mathcal Q}$, $p_i, q_i$ in the ideal generated by the arrows of
$K{\mathcal Q}$, and $p_i$ unique. Then \cite{GS} gives that $Q^3 = \bigoplus_{y \in f^3} \L \mo(y) \otimes
\mt(y) \L$ and, for $y \in f^3$ in the notation
above, the component of $A_3 (\mo(y) \otimes \mt(y))$ in the summand $\L
\mo(f_i^2) \otimes \mt(f_i^2) \L$ of $Q^2$ is $\mo(y) \otimes p_i -
q_i \otimes r_i.$

\bigskip

Given this part of the minimal projective $\L, \L$-bimodule
resolution of $\L$
$$Q^3 \stackrel{A_3}{\rightarrow} Q^2 \stackrel{A_2}{\rightarrow} Q^1
\stackrel{A_1}{\rightarrow} Q^0 \stackrel{g}{\rightarrow} \L \rightarrow
0$$
we apply ${\Hom}(-, \L)$ to give the complex
$$0 \rightarrow {\Hom}(Q^0, \L) \stackrel{d_1}{\rightarrow} {\Hom}(Q^1, \L)
\stackrel{d_2}{\rightarrow} {\Hom}(Q^2, \L) \stackrel{d_3}{\rightarrow}
{\Hom}(Q^3, \L)$$
where $d_i$ is the map induced from $A_i$ for $i = 1, 2, 3$. Then
${\HH}^2(\L) = {\Ker}\,d_3/{\Im}\,d_2.$

\bigskip

 When considering an element of the projective
$\L, \L$-bimodule $Q^1 = \bigoplus_{a, arrow} \L \mo(a) \otimes \mt(a) \L$
it is important to keep track of the individual summands of $Q^1$. So to
avoid confusion we usually denote an element in the summand $\L \mo(a)
\otimes \mt(a) \L$ by $\lambda \otimes_a \lambda'$ using the subscript `$a$'
to remind us in which summand this element lies. Similarly, an element
$\lambda \otimes_{f^2_i} \lambda'$ lies in the summand $\L \mo(f^2_i)
\otimes \mt(f^2_i) \L$ of $Q^2$ and an element $\lambda \otimes_{f^3_i}
\lambda'$ lies in the summand $\L \mo(f^3_i) \otimes \mt(f^3_i) \L$ of
$Q^3$. We keep this notation for the rest of the paper.

\bigskip

Now we are ready to compute ${\HH}^2(\L)$ for the finite dimensional algebras
${\mathcal A}_1$ and ${\mathcal A}_2$.

\section{${\HH}^2({\mathcal A}_2)$}\label{ch3}

In this section we determine ${\HH}^2({\mathcal A}_2)$ for the non-standard
algebra ${\mathcal A}_2$.

\begin{thm}
For the non-standard algebra ${\mathcal A}_2$ with $\kar K = 2$, we have $\dim\,{\HH}^2({\mathcal A}_2)= 4.$
\end{thm}

\begin{proof}
The set $f^2$ of minimal relations was given in Proposition \ref{prop1}.

Following \cite{GS, GSZ}, we may choose the set $f^3$ to consist of the
following elements:
$$\{f^3_1, f^3_2,f^3_3, f^3_4\},\mbox{ where }$$
$$\begin{array}{l c l l l l l l}
f^3_1 &=& f^2_1 \alpha \delta \gamma \beta + f^2_1 \alpha \beta \\
      &=& \alpha \delta \gamma \beta f^2_1 + \alpha \beta f^2_1 \in e_1 K{\mathcal Q}e_1, \\

f^3_2 &=& f^2_2 \xi \delta \gamma \epsilon + f^2_2 \xi \beta \alpha \epsilon  \\
      &=& \xi f^2_4 \beta \alpha \epsilon + \xi f^2_4 \delta \gamma \epsilon
      + \xi \delta \gamma f^2_4 \epsilon + \xi \beta \alpha f^2_4 \epsilon
      + \xi \delta \gamma \epsilon f^2_2 + \xi \beta \alpha \epsilon f^2_2 \in e_2 K{\mathcal Q}e_2, \\

f^3_3 &=& f^2_3 \gamma \beta \alpha \delta + f^2_3 \gamma \epsilon \xi \delta \\
      &=& \gamma f^2_4 \epsilon \xi \delta + \gamma f^2_4 \beta \alpha \delta
      + \gamma \beta \alpha f^2_4 \delta + \gamma \epsilon \xi f^2_4 \delta
      + \gamma \beta \alpha \delta f^2_3 + \gamma \epsilon \xi \delta f^2_3 \in e_3 K{\mathcal Q}e_3, \\

f^3_4 &=& f^2_4 \beta \alpha \delta \gamma + f^2_4 \epsilon \xi \delta \gamma \\
      &=& \epsilon f^2_2 \xi \delta \gamma + \delta f^2_3 \gamma \beta \alpha
      + \delta f^2_3 \gamma \epsilon \xi + \delta \gamma f^2_4 \beta \alpha
      + \delta \gamma f^2_4 \epsilon \xi \\
 &+& \beta \alpha f^2_4 \delta \gamma + \beta \alpha \delta f^2_3 \gamma
      + \delta \gamma \epsilon \xi f^2_4 + \delta \gamma \beta \alpha f^2_4 \in e_4 K{\mathcal Q}e_4.\\
\end{array}$$

Thus (writing $\Lambda$ for ${\mathcal A}_2$) the projective $Q^3 = \bigoplus_{y \in f^3} \L \mo(y) \otimes \mt(y) \L$
$= (\L e_1 \otimes e_1\L) \oplus (\L  e_2 \otimes e_2\L) \oplus
 (\L e_3 \otimes e_3\L) \oplus (\L  e_4\otimes e_4 \L) $.
We know that ${\HH}^2(\L) = {\Ker}\,d_3 / {\Im}\,d_2$. First we will find
${\Im}\,d_2$. Let $f \in {\Hom}(Q^1, \L)$ and so write
$$f(e_1 \otimes_{\alpha} e_4) = c_1 \alpha + c_2 \alpha \delta \gamma, \hspace{1cm} f(e_4
\otimes_{\beta} e_1) = c_3 \beta + c_4 \delta \gamma \beta,$$
$$f(e_3 \otimes_{\gamma} e_4) = c_5 \gamma + c_6 \gamma \beta \alpha, \hspace{1cm} f(e_4 \otimes_{\delta} e_3) = c_7 \delta + c_8 \beta \alpha \delta,$$
$$f(e_4 \otimes_{\epsilon} e_2) = c_9 \epsilon + c_{10} \delta \gamma \epsilon
\hspace{1cm} \mbox{and }
f(e_2 \otimes_{\xi} e_4) = c_{11} \xi + c_{12} \xi \delta \gamma,$$
where $c_1, c_2, c_3, c_4, \ldots, c_{12} \in K.$
Now we find $fA_2 = d_2f$. We have
$fA_2(e_1 \otimes_{f^2_1} e_1) = f(e_1 \otimes_{\alpha} e_4)\beta $+$ \alpha f(e_4 \otimes_{\beta} e_1)
-  f(e_1 \otimes_{\alpha} e_4)\delta \gamma \beta
- \alpha f(e_4 \otimes_{\delta} e_3) \gamma \beta - \alpha \delta f(e_3 \otimes_{\gamma} e_4) \beta
- \alpha \delta \gamma f(e_4 \otimes_{\beta} e_1) = c_1 \alpha \beta + c_2 \alpha \delta \gamma \beta
+ c_3 \alpha \beta + c_4 \alpha \delta \gamma \beta - c_1 \alpha \delta \gamma \beta
- c_7 \alpha \delta \gamma \beta - c_5 \alpha \delta \gamma \beta - c_3 \alpha \delta \gamma \beta
= (c_1 + c_2 + c_3 + c_4 - c_1 -c_7 - c_5 - c_3) \alpha \beta
= (c_2 + c_4 + c_7 + c_5) \alpha \beta$.

Also $fA_2(e_2 \otimes_{f^2_2} e_2) = f(e_2 \otimes_{\xi} e_4)\epsilon + \xi f(e_4 \otimes_{\epsilon} e_2)
= (c_{12} + c_{10}) \xi \delta \gamma \epsilon$.

We have $fA_2(e_3 \otimes_{f^2_3} e_3) = f(e_3 \otimes_{\gamma} e_4)\delta + \gamma f(e_4 \otimes_{\delta} e_3)
= (c_6 + c_8) \gamma \beta \alpha \delta.$

And $fA_2(e_4 \otimes_{f^2_4} e_4) = f(e_4 \otimes_{\beta} e_1)\alpha + f(e_4 \otimes_{\delta} e_3) \gamma
+ f(e_2 \otimes_{\epsilon} e_4) \xi + \beta f(e_1 \otimes_{\alpha} e_4) + \delta f(e_3 \otimes_{\gamma} e_4)
+ \epsilon f(e_2 \otimes_{\xi} e_4) $= $ c_3 \beta \alpha + c_4 \delta \gamma \beta \alpha + c_7 \delta \gamma
+ c_8 \beta \alpha \delta \gamma + c_9 \epsilon \xi + c_{10} \delta \gamma \epsilon \xi + c_1 \beta \alpha
+ c_2 \beta \alpha \delta \gamma + c_5 \delta \gamma + c_6 \delta \gamma \beta \alpha + c_{11} \epsilon \xi
+ c_{12} \epsilon \xi \delta \gamma $ = $ (c_3 + c_1) \beta \alpha + (c_7 + c_5) \delta \gamma + (c_9 + c_{11})
\epsilon \xi + (c_4 + c_2 + c_7 + c_5 + c_{10} + c_{12}) \delta \gamma \beta \alpha $= $(c_3 + c_1 + c_9 + c_{11}) \beta \alpha
+ (c_7 + c_5 + c_9 + c_{11}) \delta \gamma + (c_4 + c_2 + c_7 + c_5 + c_{10} + c_{12}) \delta \gamma \beta \alpha.$

Hence, $fA_2$ is given by
$$fA_2(e_1 \otimes_{f^2_1} e_1) = d_1 \alpha \beta,$$
$$fA_2(e_2 \otimes_{f^2_2} e_2)= d_2 \xi \delta \gamma \epsilon,$$
$$fA_2(e_3 \otimes_{f^2_3} e_3)= d_3 \gamma \beta \alpha \delta,$$
$$fA_2(e_4 \otimes_{f^2_4} e_4) = d_4 \beta \alpha + d_5 \delta \gamma + (d_1 + d_2) \delta \gamma \beta \alpha,$$
for some $d_1, \ldots, d_5 \in K$.
  So $\dim\,{\Im}\,d_2 = 5$.

Now we determine ${\Ker}\,d_3$. Let $h \in {\Ker}\,d_3$, so  $h \in
{\Hom}(Q^2, \L)$ and $d_3h = 0$. Let $h: Q^2 \rightarrow \L$ be given by

$$h(e_1 \otimes_{f^2_1} e_1) = c_1 e_1 + c_2 \alpha \delta \gamma \beta,$$
$$h(e_2 \otimes_{f^2_2} e_2) = c_3 e_2 + c_4 \xi \delta \gamma \epsilon,$$
$$h(e_3 \otimes_{f^2_3} e_3) = c_5 e_3 + c_6 \gamma \beta \alpha \delta \mbox{ and }$$
$$h(e_4 \otimes_{f^2_4} e_4) = c_7 e_4 + c_8 \beta \alpha + c_9 \delta \gamma + c_{10}  \beta \alpha \delta \gamma,$$
for some $c_1, c_2, \ldots, c_{10} \in K$.

Then $hA_3(e_1 \otimes_{f^3_1} e_1) = h(e_1 \otimes_{f^2_1}
e_1) \alpha \delta \gamma \beta + h(e_1 \otimes_{f^2_1}
e_1) \alpha \beta - \alpha \delta \gamma \beta h(e_1 \otimes_{f^2_1} e_1) - \alpha \beta h(e_1 \otimes_{f^2_1} e_1)$
$= c_1 \alpha \delta \gamma \beta + c_1 \alpha \beta - c_1 \alpha \delta \gamma \beta - c_1 \alpha \beta = 0,$

In a similar way and recalling that $\kar K = 2$, we can show that $hA_3(e_2 \otimes_{f^3_2} e_2) = 0$
and $hA_3(e_3 \otimes_{f^3_3} e_3) = 0.$

Finally, $hA_3(e_2 \otimes_{f^3_4} e_2) = h(e_4 \otimes_{f^2_4} e_4) \beta \alpha \delta \gamma +
 h(e_4 \otimes_{f^2_4} e_4) \epsilon \xi \delta \gamma
- \epsilon h(e_2 \otimes_{f^2_2} e_2) \xi \delta \gamma - \delta h(e_3 \otimes_{f^2_3} e_3) \gamma \beta \alpha - \delta h(e_3 \otimes_{f^2_3} e_3)
\gamma \epsilon \xi - \delta \gamma h(e_4 \otimes_{f^2_4} e_4) \beta \alpha - \delta \gamma h(e_4 \otimes_{f^2_4} e_4) \epsilon \xi
- \beta \alpha h(e_4 \otimes_{f^2_4} e_4) \delta \gamma - \beta \alpha \delta h(e_3 \otimes_{f^2_3} e_3) \gamma - \delta \gamma \epsilon \xi
h(e_4 \otimes_{f^2_4} e_4) - \delta \gamma \beta \alpha h(e_4 \otimes_{f^2_4} e_4)$
$ = c_7 \beta \alpha \delta \gamma  + c_7 \epsilon \xi \delta \gamma - c_3 \epsilon \xi \delta \gamma - c_5 \delta
\gamma \beta \alpha  - c_5 \delta \gamma \epsilon \xi - c_7 \delta \gamma \beta \alpha - c_7 \delta \gamma \epsilon \xi
- c_7 \beta \alpha \delta \gamma - c_5 \delta \gamma \beta \alpha - c_7 \delta \gamma \epsilon \xi - c_7 \delta \gamma \beta \alpha$
$= (c_7 - c_3 - c_5) \epsilon \xi \delta \gamma.$ As $h \in {\Ker}\,d_3$ we have $c_7 = c_3 + c_5.$

Thus $h$ is given by
$$h(e_1 \otimes_{f^2_1} e_1) = c_1 e_1 + c_2 \alpha \delta \gamma \beta,$$
$$h(e_2 \otimes_{f^2_2} e_2) = c_3 e_2 + c_4 \xi \delta \gamma \epsilon,$$
$$h(e_3 \otimes_{f^2_3} e_3) = c_5 e_3 + c_6 \gamma \beta \alpha \delta \mbox{ and }$$
$$h(e_4 \otimes_{f^2_4} e_4) = (c_3 + c_5) e_4 + c_8 \beta \alpha + c_9 \delta \gamma
+ c_{10}  \beta \alpha \delta \gamma.$$
Hence $\dim\,{\Ker}\,d_3 = 9.$

Therefore, $\dim\,{\HH}^2({\mathcal A}_2) = \dim\,{\Ker}\,d_3 - \dim\,{\Im}\,d_2 = 9 - 5
= 4.$
\end{proof}

\section{${\HH}^2({\mathcal A}_1)$}\label{ch4}
In this section we determine ${\HH}^2({\mathcal A}_1)$ for the standard
algebra ${\mathcal A}_1$.

\begin{thm}
For the standard algebra ${\mathcal A}_1$ with $\kar K = 2$, we have $\dim\,{\HH}^2({\mathcal A}_1) = 3.$
\end{thm}

\begin{proof}
The set $f^2$ of minimal relations was given in Proposition \ref{prop2}.
Following \cite{GS, GSZ}, we may choose the set $f^3$ to consist of the
following elements:
$$\{f^3_1, f^3_2,f^3_3, f^3_4\},\mbox{ where }$$
$$\begin{array}{l c l c l l l l l l l l l }
f^3_1 &=& f^2_1 \alpha \epsilon \xi \beta  \\
      &=& \alpha f^2_4 \epsilon \xi \beta
      &+& \alpha \delta \gamma f^2_4 \beta
      &+& \alpha \delta \gamma \beta f^2_1
      &+& \alpha \delta f^2_3 \gamma \beta
      &+& \alpha \epsilon f^2_2 \xi \beta
     &\in& e_1 K{\mathcal Q}e_1, \\

f^3_2 &=& f^2_2 \xi \delta \gamma \epsilon \\
      &=& \xi f^2_4 \delta \gamma \epsilon
      &+& \xi \beta \alpha f^2_4 \epsilon
      &+&  \xi \beta f^2_1 \alpha \epsilon
      &+& \xi \beta \alpha \epsilon f^2_2
      &+& \xi \delta f^2_3 \gamma \epsilon
      &\in& e_2 K{\mathcal Q}e_2, \\

f^3_3 &=& f^2_3 \gamma \epsilon \xi \delta \\
      &=& \gamma f^2_4 \epsilon \xi \delta
      &+& \gamma \beta \alpha f^2_4 \delta
      &+& \gamma \beta f^2_1 \alpha \delta
      &+& \gamma \beta \alpha \delta f^2_3
      &+& \gamma \epsilon f^2_2 \xi \delta
      &\in& e_3 K{\mathcal Q}e_3, \\

f^3_4 &=& f^2_4 \beta \alpha \delta \gamma \\
      &=& \beta f^2_1 \alpha \delta \gamma
      &+& \delta f^2_3 \gamma \epsilon \xi
      &+& \epsilon f^2_2 \xi \delta \gamma
      &+& \delta \gamma f^2_4 \epsilon \xi
      &+& \epsilon \xi f^2_4 \delta \gamma \\
      &+& \delta \gamma \beta f^2_1 \alpha
      &+& \delta \gamma \epsilon f^2_2 \xi
      &+& \epsilon \xi \delta f^2_3 \gamma
      &+& \delta \gamma \beta \alpha f^2_4
   & & &\in& e_4 K{\mathcal Q}e_4.\\
\end{array}$$
Thus (writing $\Lambda$ for ${\mathcal A}_1$) the projective
 $Q^3 = \bigoplus_{y \in f^3} \L \mo(y) \otimes \mt(y) \L$ $= (\L e_1 \otimes e_1\L)
\oplus (\L  e_2 \otimes e_2\L) \oplus (\L e_3 \otimes e_3\L) \oplus (\L  e_4\otimes e_4 \L) $.

Again, ${\HH}^2(\L) = {\Ker}\,d_3 / {\Im}\,d_2$. First we will find
${\Im}\,d_2$. Let $f \in {\Hom}(Q^1, \L)$ and so write
$$f(e_1 \otimes_{\alpha} e_4) = c_1 \alpha + c_2 \alpha \delta \gamma, \hspace{1cm} f(e_4
\otimes_{\beta} e_1) = c_3 \beta + c_4 \delta \gamma \beta,$$
$$f(e_3 \otimes_{\gamma} e_4) = c_5 \gamma + c_6 \gamma \beta \alpha, \hspace{1cm} f(e_4 \otimes_{\delta} e_3) =
 c_7 \delta + c_8 \beta \alpha \delta,$$
$$f(e_4 \otimes_{\epsilon} e_2) = c_9 \epsilon + c_{10} \delta \gamma \epsilon
\hspace{1cm} \mbox{and }
f(e_2 \otimes_{\xi} e_4) = c_{11} \xi + c_{12} \xi \delta \gamma,$$
where $c_1, c_2, c_3, c_4, \ldots, c_{12} \in K.$
Now we find $fA_2 = d_2f$. We have
$fA_2(e_1 \otimes_{f^2_1} e_1) = f(e_1 \otimes_{\alpha} e_4)\beta
$+$ \alpha f(e_4 \otimes_{\beta} e_1)
= c_2 \alpha \delta \gamma \beta
+ c_4 \alpha \delta \gamma \beta
= (c_2 + c_4 ) \alpha \delta \gamma \beta$.

Also $fA_2(e_2 \otimes_{f^2_2} e_2) = f(e_2 \otimes_{\xi} e_4)\epsilon + \xi f(e_4 \otimes_{\epsilon} e_2)
= (c_{12} + c_{10}) \xi \delta \gamma \epsilon$ and
$fA_2(e_3 \otimes_{f^2_3} e_3) = f(e_3 \otimes_{\gamma} e_4)\delta + \gamma f(e_4 \otimes_{\delta} e_3)
= (c_6 + c_8) \gamma \beta \alpha \delta$.

And $fA_2(e_4 \otimes_{f^2_4} e_4) = f(e_4 \otimes_{\beta} e_1)\alpha + f(e_4 \otimes_{\delta} e_3) \gamma
+ f(e_2 \otimes_{\epsilon} e_4) \xi + \beta f(e_1 \otimes_{\alpha} e_4) + \delta f(e_3 \otimes_{\gamma} e_4)
+ \epsilon f(e_2 \otimes_{\xi} e_4) $= $(c_3 + c_9 + c_1 + c_{11}) \beta \alpha
+ (c_7 + c_9 + c_5 + c_{11}) \delta \gamma + (c_4 + c_8 + c_{10} + c_2 + c_6 + c_{12}) \delta \gamma \beta \alpha$.
Hence, $fA_2$ is given by
$$fA_2(e_1 \otimes_{f^2_1} e_1) = d_1 \alpha \beta,$$
$$fA_2(e_2 \otimes_{f^2_2} e_2)= d_2 \xi \delta \gamma \epsilon,$$
$$fA_2(e_3 \otimes_{f^2_3} e_3)= d_3 \gamma \beta \alpha \delta,$$
$$fA_2(e_4 \otimes_{f^2_4} e_4) = d_4 \beta \alpha + d_5 \delta \gamma
+ (d_1 + d_2 + d_3) \delta \gamma \beta \alpha,$$
for some $d_1, \ldots, d_5 \in K$.
  So $\dim\,{\Im}\,d_2 = 5$.

Now we determine ${\Ker}\,d_3$. Let $h \in {\Ker}\,d_3$, so  $h \in
{\Hom}(Q^2, \L)$ and $d_3h = 0$. Let $h: Q^2 \rightarrow \L$ be given by

$$h(e_1 \otimes_{f^2_1} e_1) = c_1 e_1 + c_2 \alpha \delta \gamma \beta,$$
$$h(e_2 \otimes_{f^2_2} e_2) = c_3 e_2 + c_4 \xi \delta \gamma \epsilon,$$
$$h(e_3 \otimes_{f^2_3} e_3) = c_5 e_3 + c_6 \gamma \beta \alpha \delta \mbox{ and }$$
$$h(e_4 \otimes_{f^2_4} e_4) = c_7 e_4 + c_8 \beta \alpha + c_9 \delta \gamma
+ c_{10}  \beta \alpha \delta \gamma,$$
for some $c_1, c_2, \ldots, c_{10} \in K$.

It can be easily shown that $hA_3(e_1 \otimes_{f^3_1} e_1) = (-c_5-c_3) \alpha \delta \gamma \beta.$
As $h \in {\Ker}\,d_3$ and $\kar K =2$ we have $c_5 = c_3$,
and that $hA_3(e_2 \otimes_{f^3_2} e_2) = (-c_1 - c_5) \xi \delta \gamma \epsilon$ so that $c_1 = c_5$.
Similarly, $hA_3(e_3 \otimes_{f^3_3} e_3) = (-c_1 - c_3)\gamma \beta \alpha \delta$ so that $c_1 = c_3$.
Finally, we have $hA_3(e_2 \otimes_{f^3_4} e_2) = 0.$

Thus $h$ is given by
$$h(e_1 \otimes_{f^2_1} e_1) = c_1 e_1 + c_2 \alpha \delta \gamma \beta,$$
$$h(e_2 \otimes_{f^2_2} e_2) = c_1 e_2 + c_4 \xi \delta \gamma \epsilon,$$
$$h(e_3 \otimes_{f^2_3} e_3) = c_1 e_3 + c_6 \gamma \beta \alpha \delta \mbox{ and }$$
$$h(e_4 \otimes_{f^2_4} e_4) = c_7 e_4 + c_8 \beta \alpha + c_9 \delta \gamma
+ c_{10}  \beta \alpha \delta \gamma.$$
Hence $\dim\,{\Ker}\,d_3 = 8.$

Therefore  $\dim\,{\HH}^2({\mathcal A}_1) = \dim\,{\Ker}\,d_3 - \dim\,{\Im}\,d_2 = 8 - 5
= 3.$
\end{proof}

Thus we have shown that $\dim\HH^2({\mathcal A}_1) \neq
\dim\HH^2({\mathcal A}_2)$. Hence these two algebras are not derived
equivalent. Now we state the main result of this paper.

\begin{cor}\label{cor}
For the finite dimensional algebras ${\mathcal A}_1$ and ${\mathcal A}_2$ over an algebraically closed field $K$ with $\kar K = 2$,
we have $\dim\HH^2({\mathcal A}_1) \neq \dim\HH^2({\mathcal A}_2)$. Hence these two algebras are not derived
equivalent.
\end{cor}


\begin{thebibliography}{99}             %Any other two digit number will do

\bibitem{D} Al-Kadi, D.\, \emph{Self-injective algebras and the second Hochschild
Cohomology group},  J.\ Algebra \textbf{321} (2009), 1049--1078.

\bibitem{ES2} Erdmann, K.\ and Snashall, N., \emph{Preprojective Algebras of
Dynkin type, Periodicity and the second Hochschild Cohomology}, Algebras and
Modules II, 183--193, CMS Conf.\ Proc.\ {\bf 24}, Amer.\ Math.\ Soc., 1998.

\bibitem{GS} Green, E.\ L.\ and Snashall, N., \emph{Projective Bimodule
Resolutions of an Algebra and Vanishing of the Second Hochschild Cohomology
Group}, Forum Math. \textbf{16} (2004), 17--36.

\bibitem{GSZ} Green, E.\ L., Solberg, \O.\ and Zacharia, D., \emph{Minimal
Projective Resolutions}, Trans.\ Amer.\ Math.\ Soc. \textbf{353}
(2001), 2915--2939.

\bibitem{HS} Holm, T.\ and Skowro\'nski, A., \emph{Derived equivalence
classification of symmetric algebras of polynomial growth}, preprint, 2009.

\end{thebibliography}
\end{document}